\documentclass[11pt]{amsart}

\usepackage[margin=1.2in]{geometry}
\usepackage{hyperref}
\usepackage{amsfonts}
\usepackage{amsmath}
\usepackage{amsthm}
\usepackage{amssymb}
\usepackage{float}
\usepackage{enumerate}
\usepackage{graphicx}
\usepackage{mathrsfs}
\usepackage{xcolor}

\newtheorem{theor}{Theorem}[section]

\newtheorem{lemma}[theor]{Lemma}
\newtheorem{prop}[theor]{Proposition}

\theoremstyle{definition}
\newtheorem{rmk}[theor]{Remark}

\usepackage{enumitem}

\newcommand{\st}{\, | \,}
\newcommand{\ms}{\mathcal{M}}

\DeclareMathOperator{\inter}{\it{i}}
\DeclareMathOperator{\sys}{sys}
\DeclareMathOperator{\arcsinh}{arcsinh}
\DeclareMathOperator{\arccosh}{arccosh}
\DeclareMathOperator{\kiss}{Kiss}
\DeclareMathOperator{\area}{area}
\DeclareMathOperator{\PSL}{PSL}

\newcommand{\Z}{{\mathbb Z}}
\newcommand{\R}{{\mathbb R}}
\newcommand{\Hyp}{{\mathbb H}}

\title[Systoles and kissing numbers of finite area hyperbolic surfaces]{Systoles and kissing numbers\\of finite area hyperbolic surfaces}
\author{Federica Fanoni}
\address[Federica Fanoni]{Department of Mathematics, University of Fribourg, Switzerland}
\email{federica.fanoni@gmail.com}

\author{Hugo Parlier}
\address[Hugo Parlier]{Department of Mathematics, University of Fribourg, Switzerland}
\email{hugo.parlier@unifr.ch}
\date{\today}
\thanks{Research supported by Swiss National Science Foundation grant number PP00P2\_128557.}

\begin{document}
\begin{abstract}
We study the number and the length of systoles on complete finite area orientable hyperbolic surfaces. In particular, we prove upper bounds on the number of systoles that a surface can have (the so-called kissing number for hyperbolic surfaces). Our main result is a bound which only depends on the topology of the surface and which grows subquadratically in the genus. 
\end{abstract}
\maketitle
\section{Introduction}
In analogy with classical sphere packing problems in $\R^n$, Schmutz Schaller named and studied {\it kissing numbers} for hyperbolic surfaces. This is a particular instance of a more general analogy between the study of $n$-dimensional lattices (and their parameter spaces) and the study of hyperbolic surfaces (and their parameter spaces). Both are natural generalizations of the study of $2$-dimensional flat tori. The natural parameter spaces of these tori are $\Hyp$ and $\Hyp/ \PSL_2(\Z)$; their higher dimensional analogues include on the one hand the spaces of lattices and on the other Teichm\"uller and moduli spaces.

The classical kissing number problem is to bound the number of disjoint open unit balls that can be tangent to a fixed unit ball -- the {\it lattice kissing number} is the same problem but where one asks that the centers of the spheres lie on some lattice. This is in fact an equivalent problem to counting the number of systoles {(up to isotopy)} of the underlying lattice. Another classical topic for flat tori is the study of Hermite constants. This involves finding sharp upper bounds on the length of shortest nontrivial lattice vectors or in other words, bounds on the systole length of the quotient tori. Both of these problems make perfect sense for finite area hyperbolic surfaces and have been studied by a variety of authors including Bavard \cite{Bavard2} and Schmutz Schaller \cite{schmutz1,schmutz3}. 

Schmutz Schaller provided a variety of results on the length and the number of systoles on both closed hyperbolic surfaces and finite area complete. Lower bounds for either of these quantities can be found using arithmetic methods. Buser and Sarnak \cite{BS} were the first to show that there exist families $S_k$ of closed surfaces of genus $g_k$ with $g_k \to \infty$ as $k\to \infty$ whose systole length grows like
$$
\sys(S_k) \geq \frac{4}{3} \log g_k.
$$
Katz, Schaps and Vishne \cite{katzetal} generalized this construction to principal congruence subgroups of arbitrary arithmetic surfaces. Makisumi \cite{makisumi} showed that, in some sense, this is best one can hope for via arithmetic constructions. Schmutz Schaller \cite{schmutz3} found analogous results for kissing numbers: for any $\varepsilon>0$, there is a family of closed surfaces $T_k$ of genus $g_k$ with $g_k \to \infty$ as $k\to \infty$ whose number of systoles grows like
$$
\kiss(T_k) \geq g_k^{\frac{4}{3}-\varepsilon}.
$$
For surfaces with cusps, families reaching these lower bounds (for both quantities) are directly obtainable by considering principal congruence subgroups of $\PSL_2(\Z)$ (see \cite{schmutz4, brooks, BMP}). The number of cusps in these examples grows roughly like $g^{\frac{2}{3}}$. 

Upper bounds for these quantities have also been studied, in particular for closed surfaces. Via an easy area argument, one can obtain an upper bound on systole length of closed surfaces of genus $g$ that grows like $2 \log g$. This complements Buser and Sarnak's lower bound to show that the rough growth is logarithmic but the discrepancy between the $\frac{4}{3}$ and the $2$ remains mysterious. Schmutz Schaller, using a disk packing argument of Fejes T\'oth, proved a very nice upper bound on systole length which is actually sharp for the congruence subgroups of $\PSL_2(\Z)$ (see also \cite{Adams, Bavard1}). We use this result in an essential way and give the exact formulation in the sequel (Theorem \ref{schmutzsys}).

For kissing numbers, the best known upper bounds are results of the second author \cite{parlier}. In particular it is shown that there is a bound which depends only on the genus $g$, and which grows at most subquadratically in function of $g$. Again, there is a discrepancy between the $g^{\frac{4}{3}}$ lower bound and the $g^2$ upper bound (although the latter cannot be sharp). Upper bounds for kissing numbers of non closed finite area complete surfaces (i.e.\ surfaces with cusps) have yet to be approached. Filling this gap is the main goal of our article. 

One of the main consequences of what we obtain is the following.

\begin{theor}\label{thm:intro1}
There exists a universal constant $C>0$ such that for any $S\in \ms_{g,n}$, $g\geq 1$, its kissing number satisfies
$$\kiss(S)\leq C\,(g+n)\frac{g}{\log(g+1)}.$$
\end{theor}

We obtain this result as a consequence of a number of results concerning the length and the topological configurations of systoles.

In particular, concerning the length of systoles, we show the following. 
\begin{theor}\label{thm:intro2}
There exists a universal constant $K<8$ such that every $S\in\ms_{g,n}$ ($g\neq 0$) satisfies
$$\sys(S)\leq 2\log(g)+K.$$
\end{theor}
The result is not surprising in view of the results for closed surfaces and Schmutz Schaller's bound, but it is interesting to note that it is asymptotically a stronger bound when the growth of the number of cusps is bounded above by $g^{\frac{1}{2}}$.

Our results on topological configurations of systoles can be summarized as follows. 
\begin{theor}\label{thm:intro3}
If $\alpha$ and $\beta$ are systoles of a surface $S\in  \ms_{g,n}$, then
$$i(\alpha,\beta)\leq 2$$
and if $i(\alpha,\beta)= 2$, then either $\alpha$ or $\beta$ surrounds two cusps. Furthermore, for every genus $g\geq 0$, there exists $n(g)\in \mathbb{N}$ and a surface $S_g$ of genus $g$ with $n(g)$ cusps which has systoles that intersect twice.
\end{theor}
The above result is in contrast with closed surfaces where systoles can intersect at most once.

Finally, we obtain the following bound which relates systolic lengths and kissing numbers. 
\begin{theor}\label{thm:intro4}
If $S\in \ms_{g,n}$ has systole of length $\sys(S)=\ell$, then
$$\kiss(S)\leq 20\, n\cosh(\ell/4)+200\,\frac{e^{\ell/2}}{\ell}(2g-2+n).$$
\end{theor}

The article is organized as follows. In Section \ref{sec:length} we prove our upper bounds on systole length. Section \ref{sec:topology} is dedicated to the study of the topological configurations of systoles. In the final section we prove Theorems \ref{thm:intro4} and  \ref{thm:intro1}.

{\bf  Acknowledgements.}
{The authors would like to thank Bram Petri for useful discussions and the reviewer for her or his pertinent remarks and suggestions. Both authors were fully supported by Swiss National Science Foundation grant number PP00P2\_128557. In addition, the authors acknowledge support from U.S. National Science Foundation grants DMS 1107452, 1107263, 1107367 "RNMS: GEometric structures And Representation varieties" (the GEAR Network).}

\section{Bounds on lengths of systoles}\label{sec:length}
We denote by $\ms_{g,n}$ the {\it moduli space} of surfaces of signature $(g,n)$, by which we mean the space of all complete, finite area hyperbolic surfaces of genus $g$ with $n$ cusps up to isometry. We shall always consider that $g$ and $n$ satisfy $3g-3 + n >0$. A {\it systole} of a surface $S\in \ms_{g,n}$ is a shortest closed geodesic. We think of systoles -- and closed geodesics in general -- as being non-oriented. Given a surface $S$, we denote its systole length (the length of one of its systoles) by $\sys(S)$. The main objective of this section is to show that every surface of genus $g\geq 1$, with or without cusps, has systole length bounded above by a function which only depends on the genus.

For any cusp $c$, let $H_c$ be the associated open horoball region of area $2$. By the  collar lemma (see for instance \cite[Chapter 4]{buser}), two such regions are disjoint.

For any cusp $c$ and any non-negative $r$, define the set $D_r(c)$ to be
$$D_r(c):=\{p\in S\st d(p,H_c)< r\}\cup H_c.$$
If $D_r(c)$ is homeomorphic to a once-punctured disk, we can compute its area, which is
$$\area(D_r(c))=2e^r.$$
\begin{lemma}\label{twobounds}\ \\
(a) If there are two cusps $c$ and $c'$ such that $D_r(c)$ and $D_r(c')$ are tangent, then the simple closed geodesic forming a pair of pants with them has length $4\arccosh(e^ r)$, so
$$\sys(S)\leq 4\arccosh(e^r).$$
(b) If $D_r(c)$ is tangent to itself for some $r\geq \log(2)$, then $$\sys(S)\leq 2\arccosh(e^r-1).$$

\end{lemma}

\begin{proof}
(a) Consider the pair of pants determined by the two cusps and the simple closed geodesic $\gamma$ surrounding them. Cut it along the orthogonal from $\gamma$ to itself, the shortest geodesic between the cusps and the perpendiculars from the cusps to $\gamma$. Consider one of the four obtained quadrilaterals; we denote its vertices by $q$, $s$, $t$ and $c$ and the intersection point of $\partial H_c$ with a side by $p$, as in Figure \ref{lambert}.
\begin{figure}[H]
\begin{center}
\includegraphics{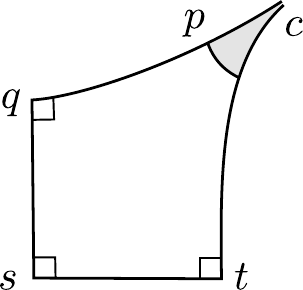}
\caption{One of the quadrilaterals}\label{lambert}
\end{center}
\end{figure}
Draw the quadrilateral in the upper half-plane, choosing infinity as ideal point. We fix the two geodesics containing $qc$ and $tc$ to be $x=0$ and $x=1$. The area of $H_c$ intersected with the quadrilateral is $1$, so $\partial H_c$ is given by $y=1$ and $p=i$. Moreover, $d(p,q)=\frac{1}{2}d(H_c,H_{c'})=r$, so $q=ie^{-r}$. Consider $\mathcal{C}_1$ and $\mathcal{C}_2$, the Euclidean circles representing the geodesics through $q$ and $s$ and through $s$ and $t$.
\begin{figure}[H]
\begin{center}
\includegraphics{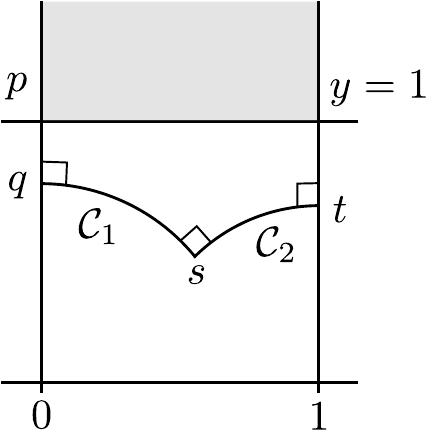}\label{upperhalflambert}
\caption{In the upper half-plane}
\end{center}
\end{figure}
Since $\mathcal{C}_1\perp \{x=0\}$, $\mathcal{C}_2\perp \{x=1\}$ and $\mathcal{C}_1\perp \mathcal{C}_2$, they have equations
$$\mathcal{C}_1: x^2+y^2=R^2$$
and
$$\mathcal{C}_2: (x-1)^2+y^2=1-R^2$$
for some $R$. As $q\in \mathcal{C}_1$, we have $R=e^r$. By imposing $d(t,s)=\ell/4$, we obtain $\ell=4\arccosh(e^r)$.

(b) The cusp $c$ with the curve of length $2r$ from $H_c$ and back determines a pair of pants with at least one simple closed geodesic as boundary. 

If the pair of pants has two cusps and a boundary curve $\alpha$, we can cut it along the geodesic between the two cusps, the shortest geodesics between the cusps and $\alpha$ and the geodesic containing curve of length $2r$. We get two right-angled triangles with two ideal vertices and $\frac{\pi}{2}$ and two quadrilaterals with three right angles and an ideal vertex, as in Figure \ref{case2withcusp}.
\begin{figure}[H]
\begin{center}
\includegraphics{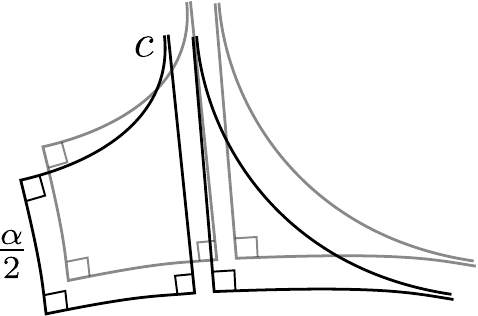}
\caption{The cut pair of pants with two cusps}\label{case2withcusp}
\end{center}
\end{figure}
By direct computation similar to before, we obtain $$\ell(\alpha)=2\arccosh(e^r-1).$$

If the pair of pants has two boundary curves we denote them $\alpha$ and $\beta$, and we suppose that $\ell(\alpha)\leq \ell(\beta)$. We cut along the orthogonal from $\alpha$ to $\beta$, the shortest geodesics from $\alpha$ and $\beta$ to the horoball and the geodesic containing the curve of length $2r$. We obtain four quadrilaterals, with three right angles and an ideal vertex, two by two isometric.
\begin{figure}[H]
\begin{center}
\includegraphics{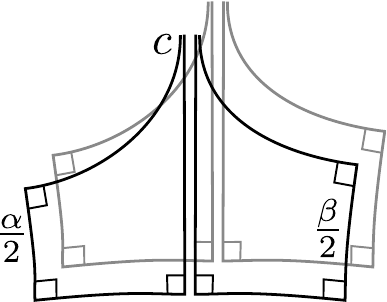}
\caption{The cut pair of pants with one cusp}
\end{center}
\end{figure}
Again by direct computation we have
$$\ell(\alpha)=2\arccosh(ae^r)$$
and
$$\ell(\beta)=2\arccosh((1-a)e^r)$$
where $a$ is the area of $H_c$ intersected with one of the two quadrilaterals containing a part of $\alpha$. Since  $\ell(\alpha)\leq \ell(\beta)$, we have $a\leq \frac{1}{2}$. Moreover, $\alpha$ is longest when $a$ is maximum, that is when $a=\frac{1}{2}$. In this case
$$\ell(\alpha)=\ell(\beta)=2\arccosh\left(\frac{e^r}{2}\right).$$
Since by assumption $r\geq\log(2)$, we get that in both cases the curve $\alpha$ satisfies $$\ell(\alpha)\leq2\arccosh(e^r-1).$$
\end{proof}

\begin{rmk}\label{leqlog2}
From the proof of the lemma we also have that if $D_r(c)$ is tangent to itself for some $r\leq \log(2)$, then $\sys(S)\leq 2\arcsinh(1)$.
\end{rmk}

We can now prove our bound on systole length for surfaces of genus $g\geq 1$.

\begin{theor}\label{sysbound}
There exists a universal constant $K<8$ such that every $S\in\ms_{g,n}$ satisfies
$$\sys(S)\leq 2\log(g)+K.$$
\end{theor}

\begin{proof}
Set $\ell=\sys(S)$. 

We begin with recalling the well known situation where $n=0$ (and thus $g\geq 2$). As the surface is closed, any open disk $D_{\ell/2}(p)$ of radius $\ell/2$ is embedded in the surface and thus
$$\area(D_{\ell/2}(p))=2\pi(\cosh(\ell/2)-1)\leq \area(S)=2\pi(2g-2)$$
which in turn implies
$$\ell\leq 2\log(g)+2\log(4).$$

Suppose now that $n\geq 1$. We split the proof into three non-mutually exclusive cases. The first situation we consider is when there are ``many'' cusps (how many will be made explicit): in this case two $D_c(r)$'s have to meet for a ``small" $r$ and will determine a short curve. In the second case, we assume that there are two cusps which are close to each other and the systole length will be bounded by the length of the curve surrounding them. In the final situation there are ``few'' cusps and we further assume any two are far away: in this case we show that there is a cusp with a short loop from its horoball to itself which in turn determines a short curve.

\underline{Case 1:} $n\geq \sqrt{2\pi g}$

If the sets $D_r(c)$ are pairwise disjoint for different cusps $c$ and each homeomorphic to a once-punctured disk, then
$$\area\left(\bigcup_{c \mbox{  \scriptsize cusp}}D_r(c)\right)=2ne^r\leq \area(S)=2\pi(2g+n-2)$$
thus
$$e^r\leq \frac{\pi(2g-2+n)}{n}.$$

Since $n\geq \sqrt{2\pi g}$, this implies $$e^r\leq\frac{\sqrt{2\pi}(g-1)}{\sqrt{g}}+\pi.$$
So for some $r\leq \log \left(\frac{\sqrt{2\pi}(g-1)}{\sqrt{g}}+\pi\right)$ either two $D_r(c)$ are tangent to each other or one is tangent to itself. Lemma \ref{twobounds} now implies 
$$\ell\leq 4\arccosh\left(\frac{\sqrt{2\pi}(g-1)}{\sqrt{g}}+\pi\right).$$

\underline{Case 2:} there are two distinct cusps $c_1$ and $c_2$ with $d(H_{c_1},H_{c_2})\leq\log(2\pi(g-1+\sqrt{2\pi g}))$

By Lemma \ref{twobounds}
$$\ell\leq 4\arccosh\left(\sqrt{2\pi(g-1+\sqrt{2\pi g})}\right)$$
and we are done.

\underline{Case 3:} $0<n<\sqrt{2\pi g}$ and for any two cusps $c_1$, $c_2$ satisfy 
$$d(H_{c_1},H_{c_2})>\log(2\pi(g-1+\sqrt{2\pi g})).$$

We fix a cusp $c$. Since any two cusps are far away, for $r\leq \log(2\pi(g-1+\sqrt{2\pi g}))$ the set$D_r(c)$ is disjoint from any other $H_{c'}$. If it is also an embedded once-punctured disk, then
$$\area(D_r(c))=2e^r\leq\area(S)<4\pi(g-1+\sqrt{2\pi g})$$
so
$$e^r\leq \log(2\pi(g-1+\sqrt{2\pi g})).$$
We deduce that for some $r\leq \log(2\pi(g-1+\sqrt{2\pi g}))$, $D_r(c)$ is tangent to itself. By Remark \ref{leqlog2}, if $r\leq \log(2)$ then $\ell\leq 2\arcsinh(1)$. Otherwise, by Lemma \ref{twobounds}, we obtain
$$\ell\leq 2\arccosh(2\pi(g-1+\sqrt{2\pi g}))-1).$$

Now any surface with $n>0$ will be in one of the three cases detailed above and as such we can deduce:
\begin{gather*}\ell\leq \max\left\{4\arccosh\left(\sqrt{2\pi}(g-1)/\sqrt{g}+\pi\right),4\arccosh\left(\sqrt{2\pi(g-1+\sqrt{2\pi g})}\right),\right.\\ \left. 2\arccosh(2\pi(g-1+\sqrt{2\pi g}))-1)\right\} < 2\log(g)+8.
\end{gather*}
\end{proof}

Applying the techniques of the above theorem to punctured spheres, one can show that the systole length of punctured sphere is bounded by a uniform constant (which doesn't depend on the number of cusps). This is also a consequence of a theorem of Schmutz Schaller who provided a different bound for the systole length of punctured surfaces.
\begin{theor}[\cite{schmutz1}]\label{schmutzsys}
For $S\in \ms_{g,n}$, with $n\geq 2$ we have
$$\sys(S)\leq 4\arccosh\left(\frac{6g-6+3n}{n}\right).$$
\end{theor}
For $n\sim g^\alpha$, Schmutz Schaller's bound grows roughly like $4(1-\alpha)\log(g)$. So our bound is stronger for $\alpha<\frac{1}{2}$, while Schmutz Schaller's is better for $\alpha\geq \frac{1}{2}$.

\section{Intersection properties of systoles}\label{sec:topology} 

It is well known, via a simple cutting and pasting argument, that systoles on closed surfaces pairwise intersect at most once. On surfaces with cusps, this is not necessarily the case. For instance, on punctured spheres, it is not difficult to see that systoles can intersect twice (the simplest case is a four times punctured sphere with at least two systoles - they necessarily intersect and the minimal intersection number between two distinct curves is $2$). This phenomenon also occurs for surfaces with positive genus. An example of this can be derived from Buser's hairy torus (cf.\ \cite[Chapter 5]{buser}) with cusps instead of boundary curves and explicit examples in all genera are given in the sequel. On the other hand, since systole length is bounded within each moduli space, it follows from the  collar lemma that the intersection number between any two systoles is also bounded. This can be considerably sharpened: the first main result of this section will be that two systoles on punctured surfaces can intersect at most twice. 

We begin with some notation and well known preliminary results. A curve is {\it non-trivial} if it represents a non-trivial element of the fundamental group. A non-trivial curve is {\it essential} if it does not bound a cusp. In particular, systoles are the shortest essential curves of a surface. Given two closed curves $\alpha$ and $\beta$, we denote by $\inter(\alpha, \beta)$ their {\it geometric intersection number} (the minimum number of transversal intersection points among representatives in the isotopy classes $[\alpha]$ and $[\beta]$). Two curves are said to intersect {\it minimally} if they intersect minimally among all representatives of their respective isotopy classes. The unique geodesics in the isotopy classes of simple closed curves are also simple and intersect minimally.

Let $\alpha$ and $\beta$ be simple closed geodesics on a surface $S$ with $\inter(\alpha, \beta)\geq 2$ and fix orientations on them. The curve $\alpha$ divides $\beta$ into arcs between consecutive intersection points. We say such an arc is of {\it type I} if the orientations at the two intersection points are different and of {\it type II} if the orientations are the same.

Note that the orientation at each intersection point depends on the choice of orientations of $\alpha$ and $\beta$, but being of type I or II is independent on the choice of orientations.
\begin{figure}[H]
\begin{center}
\includegraphics{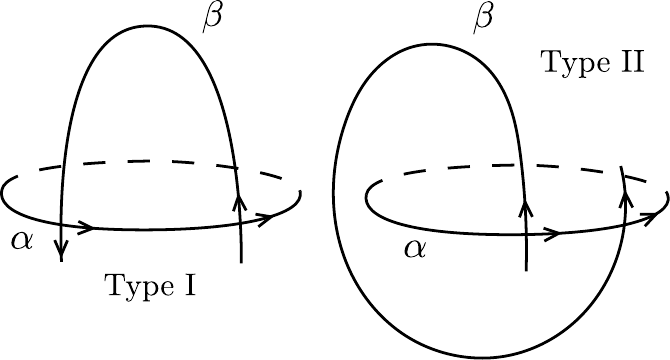}
\caption{The two kinds of arcs}
\end{center}
\end{figure}
\begin{lemma}\label{consecutive}
If $\alpha$ and $\beta$ are systoles of a surface $S\in \ms_{g,n}$ with $\inter(\alpha,\beta)\geq 2$, all arcs between consecutive intersection points are of type I.
\end{lemma}

\begin{proof}
By contradiction, suppose that $\beta$ contains arcs of type II. If there are at least two of them, there exists one, say $\beta_1$, of length at most $\frac{1}{2}\sys(S)$. Since $\beta_1$ divides $\alpha$ into two arcs, one of the two is of length at most $\frac{1}{2}\sys(S)$. Call this arc $\alpha_1$ and consider the curve $\alpha_1\cup\beta_1$.

If $\alpha_1\cup\beta_1$ were essential, its geodesic representative would be shorter than $\sys(S)$, which is impossible. Thus $\alpha_1\cup\beta_1$ must be non-essential. However one can construct a curve $\gamma$ homotopic to $\alpha_1\cup\beta_1$ such that $|\gamma\cap\alpha|=1$, so via the bigon criterion (see for instance \cite{farbmargalit})$\gamma$ and $\alpha$ intersect minimally. Thus
$$\inter(\gamma,\alpha)=1$$
and as such $\gamma$ is non-trivial in homology and is therefore essential, a contradiction.

If there is exactly one arc $\beta_1$ of type II, there should be at least two (consecutive) arcs $\beta_2$ and $\beta_3$ of type I. Then if $\ell(\beta_1)\leq \frac{1}{2}\sys(S)$, we can argue as before to obtain a contradiction. If not, then $\ell(\beta_2\cup\beta_3)\leq\frac{1}{2}\sys(S)$. The arcs $\beta_2$, $\beta_3$ and $\alpha$ determine an embedded four-holed sphere with a non-trivial curve of length at most $\sys(S)$. By construction, the geodesic in the isotopy class of this curve is strictly shorter than the systole, a contradiction.
\end{proof}
\begin{prop}
If $\alpha$ and $\beta$ are systoles of $S\in \ms_{g,n}$, then $\inter(\alpha,\beta)\leq 2$.
\end{prop}
\begin{proof}
Suppose by contradiction that $\inter(\alpha,\beta)>2$. By Lemma \ref{consecutive}, all arcs between consecutive intersection points are of type I, so $\inter(\alpha,\beta)$ is even. Thus there are at least four intersection points and at least four arcs of $\beta$ between consecutive intersection points. This implies that there is an intersection point and two arcs $\beta_1$ and $\beta_2$ departing from it with $\ell(\beta_1\cup\beta_2)\leq\frac{1}{2}\sys(S)$. We argue as in the proof of Lemma \ref{consecutive}: $\beta_1$, $\beta_2$ and $\alpha$ determine an embedded four-holed sphere with a non-trivial curve of length at most $\frac{1}{2}\sys(S)$. By construction, the geodesic in the isotopy class of this curve is strictly shorter than the systole, a contradiction.
\end{proof}

The next proposition shows that if two systoles intersect twice, there is a constraint on the topological configuration of the two curves.
\begin{prop}\label{2cusps}
If two systoles $\alpha$ and $\beta$ intersect twice, one of the two bounds two cusps.
\end{prop}

\begin{proof}
The two curves cut each other into arcs $\alpha_1$, $\alpha_2$ and $\beta_1$, $\beta_2$. Without loss of generality, we can assume $\ell(\alpha_1)\leq\ell(\beta_1)\leq\frac{1}{2}\sys(S)$. Consider $\gamma_1=\alpha_1\cup\beta_1$ and $\gamma_2=\alpha_1\cup\beta_2$. As $\gamma_1$ and $\gamma_2$ do not surround bigons, they cannot be trivial and as they can be represented by curves of length strictly less than $\sys(S)$, they must both bound a cusp. Hence $\beta$ bounds two cusps.
\end{proof}

An obvious consequence of Proposition \ref{2cusps} is that systoles on surfaces with at most one cusp intersect at most once. In the case of tori this can be improved to show that a surface with twice intersecting systoles has at least three cusps.

\begin{lemma}
If $S\in \ms_{1,2}$ and $\alpha$, $\beta$ are systoles of $S$, then $\inter(\alpha,\beta)\leq 1$.
\end{lemma}

\begin{proof}
Suppose two systoles $\alpha$ and $\beta$ intersect twice. Then $\sys(S)\geq 4\arcsinh(1)$ (see \cite{GS}) and by Proposition \ref{2cusps} one of the two curves bounds two cusps. Cut the surface along $\alpha$ and consider the one-holed torus component. The length of the shortest closed geodesic $\gamma$ in the one-holed torus which doesn't intersect $\alpha$ satisfies (see \cite{parlier2}) 
$$\cosh(\ell(\gamma)/2)\leq \cosh(\ell(\alpha)/6)+\frac{1}{2}$$
and $\ell(\gamma)\geq \sys(S)=\ell(\alpha)$, so
$$\cosh(\ell(\alpha)/2)\leq \cosh(\ell(\alpha)/6)+\frac{1}{2},$$
which contradicts $\ell(\alpha)\geq 4\arcsinh(1)$.
\end{proof}

On the other hand, we can prove that for every genus there is a punctured surface with systoles intersecting twice. The constructions will involve gluing ideal hyperbolic triangles. Any such triangle has a unique maximal embedded disk tangent to all three sides. We say that two such triangles are glued {\it without shear} if their embedded disks are tangent. 

\begin{lemma}\label{2intersections}
For every $g\geq 0$, there exists $n(g){\in \mathbb{N}}$ and a surface $S\in \ms_{g,n(g)}$ with two systoles intersecting twice.
\end{lemma}

\begin{proof}
For $g=0$, we can set $n(0)=4$, as mentioned at the beginning of section \ref{sec:topology}: any four times punctured sphere with at least two systoles will satisfy the requirement. To show the existence of such a surface, pick any $S\in \ms_{0,4}$. If it has only one systole $\gamma$, increase the length of $\gamma$ so that it is still a systole and there is another simple closed geodesic on the surface of the same length.

For $g\geq 1$, we use a building block constructed as follows. Consider a square and a triangulation of it with $32$ triangles, given by first subdividing the square into a grid of $16$ squares and then adding one diagonal for all squares as in Figure \ref{hexgrid}.
\begin{figure}[H]
\begin{center}
\includegraphics[width=2.5cm]{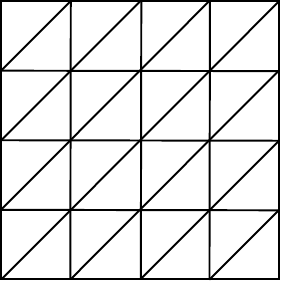}
\caption{The triangulation of the square}\label{hexgrid}
\end{center}
\end{figure}
Each of the triangles in the square will be replaced by an ideal hyperbolic triangle and all gluings will be without shear.

For $g=1$, glue opposite sides of the square (again triangles are glued without shear) to obtain a torus with $n(1)=16$ cusps.

For $g\geq 2$, consider a polygon obtained by gluing a $1\times(g-1)$ rectangle and a $1\times 2(g-1)$ rectangle along the long sides, as in Figure \ref{4ggon}.
\begin{figure}[H]
\begin{center}
\includegraphics[width=2cm]{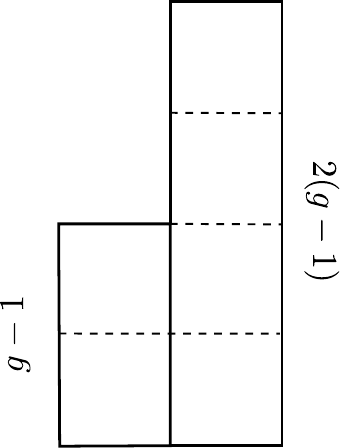}
\caption{The polygon for $g=3$}\label{4ggon}
\end{center}
\end{figure}
Think of this polygon as a $4g$-gon (with sides corresponding to sides of the squares). Fix an orientation and choose a starting side, to identify the $4g$ sides following the standard pattern $a_1b_1a_1^{-1}b_1^{-1}\dots a_gb_ga_g^{-1}b_g^{-1}$ to obtain a genus $g$ surface. If we now replace each $1\times 1$ square by the building block (always gluing adjacent triangles without shear), we get a surface of genus $g$ {with a decomposition into $32\cdot3(g-1)$ ideal triangles. Since it is a triangulation, the number of edges is $\frac{3}{2}\cdot32\cdot3(g-1)$. By an Euler characteristic argument, this implies that the surface has $n(g)=46g-46$ cusps}.

For any $g\geq 1$, consider the set $\mathcal{C}_g$ of curves surrounding pairs of cusps which are connected by an edge between vertices of degree $6$ in the triangulation of the surface. By construction, each of these intersects another such curve twice and we defer the proof that these curves are systoles to Lemma \ref{noshear}.
\end{proof}

We now prove our claim that the curves in $\mathcal{C}_g$ are indeed systoles.

\begin{lemma}\label{noshear}
For all $g\geq 1$, the curves in $\mathcal{C}_g$ are systoles.
\end{lemma}

\begin{proof}
Consider the triangulation of the surface. For $g=1$, all vertices are of degree $6$. When $g\geq 2$, the pasting scheme associates all exterior vertices of the $4g$-gon and the point in the quotient has degree $12g-6$ {(to see this simply apply the hand-shaking lemma to the graph given by the triangulation)}. The remaining vertices are all of degree $6$. We denote by $\Gamma$ the graph dual to the triangulation. From what we have just said, for $g=1$, cutting the surface along $\Gamma$ decomposes the surface into hexagons. When $g\geq 2$, cutting along $\Gamma$ decomposes the surface into  hexagons and a single $(12g-6)$-gon.

Any simple closed oriented geodesic $\gamma$ on the surface can be homotoped to a curve on $\Gamma$. At every vertex crossed by the curve, the orientations on the surface and on the curve give us a notion of ``going left'' or ``going right''. We can associate to $\gamma$ a word $w$ in the matrices $L=\left(\begin{tabular}{cc}1&1\\0&1\end{tabular}\right)$ and $R=\left(\begin{tabular}{cc}1&0\\1&1\end{tabular}\right)$, where each $L$ corresponds to a left turn and each $R$ to a right turn. This way of understanding curves on ``zero shear surfaces" is fully explained in \cite{brooksmakover}. In particular, Brooks and Makover show how to compute the length of these curves in terms of the associated word:
$$\ell(\gamma)=2\arccosh\left(\frac{\scriptsize\mbox{Tr}(w)}{2}\right).$$

Each curve in $\mathcal{C}_g$ corresponds to the word $w_0=RL^4RL^4$ (or $LR^4LR^4$ or any cyclic permutation of these, depending on the choice of an orientation and a starting point on the curve), which via a simple computation has trace $34$. To show that the curves in $\mathcal{C}_g$ are systoles, it is enough to show that all other words corresponding to simple closed geodesics have trace at least $34$.

We use the following remark (see for instance \cite{petri}):
\begin{rmk}\label{subword}
If a word can be written as $w=\dots w_1\dots w_2\dots w_k\dots$, then
$$\mbox{Tr}(w)\geq \mbox{Tr}\left(w_{\sigma(1)}\dots w_{\sigma(k)}\right)$$
for any cyclic permutation $\sigma$ of $1,\dots , k$.
\end{rmk}

Let $\gamma$ be a simple closed geodesic which is not in $\mathcal{C}_g$. First we observe that we only need to consider curves represented by circuits in $\Gamma$. Indeed, if $\gamma$ corresponds to a closed path which contains an essential (i.e. not corresponding to a curve going around a cusp) circuit $\gamma'$, a word representing $\gamma$ will contain a word representing $\gamma'$. By Remark \ref{subword}, $\gamma'$ is at most as long as $\gamma$ and we can consider $\gamma'$ instead. Otherwise, if $\gamma$ is formed from non-essential circuits, it should contain at least two of them. Note that since non-essential circuits surround a cusp, they trace a hexagon or a $(12g-6)$-gon. If both these circuits surround hexagons, we are in one of the following situations:
\begin{figure}[H]
\begin{center}
\includegraphics{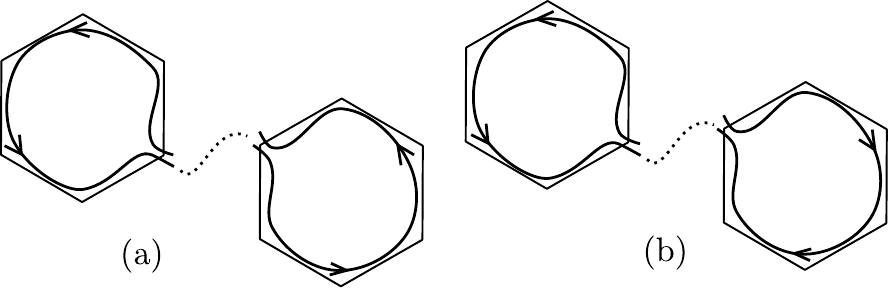}
\end{center}
\end{figure}
In case (a), a word associated to the curve contains $RL^5\dots RL^5$ and in case (b) it contains $LR^5\dots RL^5$. In both cases, by Remark \ref{subword} and a simple computation, their traces are bigger than $34$. Now if one of the two circuits surrounds the vertex of the triangulation of degree $12g-6$, the curve is even longer.

Suppose then that $\gamma$ is represented by an essential circuit. If it passes through five consecutive edges of a hexagon (said differently, a corresponding word contains $L^4$) and is not in $\mathcal{C}_g$, the following modification of the curve (see Figure \ref{shortening}) provides an essential circuit.
\begin{figure}[H]
\begin{center}
\includegraphics{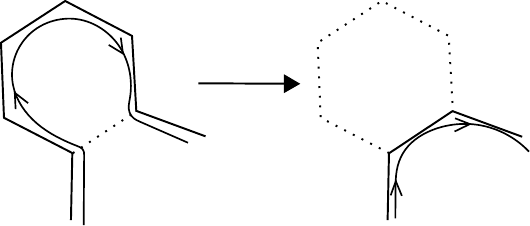}
\caption{Shortening a curve}\label{shortening}
\end{center}
\end{figure}
A word of the curve on the left contains $LR^4L$, while the one of the curve on the right contains $R^2$, so the trace decreases (again by Remark \ref{subword}) and we obtain a shorter curve.

We now assume a word $w$ representing $\gamma$ does not contain $L^ 4$ or $R^ 4$ and as such it is made of blocks of type $L^ iR^ j$, for $1\leq i,j\leq 3$. If $w$ is made of four or more such blocks, then
$$\mbox{Tr}(w)\geq \mbox{Tr}((LR)^ 4)>34.$$
Moreover, the length of $w$ is at least $7$, as the shortest circuits in $\Gamma$ are of length $6$ and correspond to curves surrounding cusps. With this in hand, one needs to check the finite set of words $w$ made of blocks as above, of length at least $7$, and of trace at most $33$. To do this one can proceed as follows. The conditions on $w$ give two systems of equations for the exponents of $L$ and $R$ (a system for the words made of two blocks as above and one for words made of three blocks). These systems can be solved to get the set of words we are interested in. It is then straightforward to check that the curves corresponding to these words do not correspond to simple closed geodesics on the surface. 
\end{proof}

\section{Kissing number bound}
In this section we will prove an upper bound for the kissing number depending on systole length. We then deduce a universal upper bound depending only on the signature of the surface. To do so, we separate the systoles into three sets and we give separate bounds for each of their cardinalities.

For a surface $S$, let $\mathfrak{S}(S)$ be the set of its systoles and $\kiss(S):=|\mathfrak{S}(S)|$ be the {\it kissing number} of $S$. We say that $\alpha$ and $\beta$ bound a cusp if they form a pair of pants with a cusp. We define:
\begin{align*}
A(S)&:=\{\alpha\in \mathfrak{S}(S)\st\alpha \mbox{ bounds two cusps}\}\\
B(S)&:=\{\alpha\in \mathfrak{S}(S)\setminus A(S)\st\exists \beta\in \mathfrak{S}(S)\setminus A(S) \mbox{ s.t. $\alpha$ and $\beta$ bound a cusp}\}\\
C(S)&:=\mathfrak{S}(S)\setminus(A(S)\cup B(S)).
\end{align*}

Note that by Proposition \ref{2cusps} two systoles in $\mathfrak{S}(S)\setminus A(S)$ intersect at most once.

\subsection{Bounds on $|A(S)|$}

As seen in Lemma \ref{twobounds}, a curve of length $\ell$ bounds two cusps $c$ and $c'$ if and only if the distance between $H_c$ and $H_{c'}$ is
$$d(\ell)=2\log(\cosh(\ell/4)),$$
To bound $|A(S)|$ we will bound the number of pairs of cusps at distance $d(\sys(S))$.
\begin{lemma} Let $S$ be a surface with $\sys(S)=\ell$ and $c$ a cusp of $S$. There are at most $\lfloor 2\cosh(\ell/4)\rfloor$ cusps $c'$ which satisfy $d(H_c,H_{c'})=d(\ell)$.
\end{lemma}
\begin{proof}
Suppose $c_1$ and $c_2$  two cusps such that
$$d(H_c,H_{c_1})=d(H_{c},H_{c_2})=d(\ell).$$
Since $\sys(S)=\ell$, the distance between $H_{c_1}$ and $H_{c_2}$ is at least $d(\ell)$. Consider
 \begin{itemize}[itemsep=2ex,leftmargin=0.5cm]
\item the segment $\alpha$ realizing the distance between $H_{c}$ and $H_{c_1}$,
\item the segment $\beta$ realizing the distance between $H_{c}$ and $H_{c_2}$,
\item the shortest arc $\gamma$ of $\partial H_c$ bounded by the endpoints of $\alpha$ and $\beta$.
\end{itemize}
Let $\delta$ be the unique geodesic segment freely homotopic with endpoints on $\partial H_{c_1}$ and $\partial H_{c_2}$ to the curve $\alpha\cup \beta\cup \gamma$. Then its length is at least $d(\ell)$.
\begin{figure}[H]
\begin{center}
\includegraphics{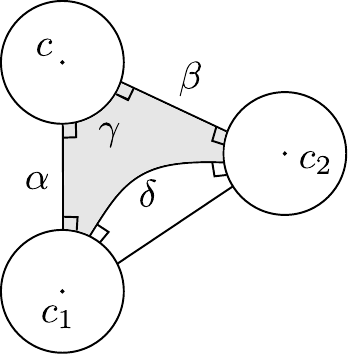}
\caption{}\label{3cusps}
\end{center}
\end{figure}
By a direct computation on the (non-geodesic) hexagon determined by $\alpha$, $\beta$, $\delta$ and the three horocycles, one can show that
$$\ell(\gamma)\geq\frac{1}{\cosh(\ell/4)}.$$
Since $\partial H_c$ has length $2$, the number of cusps around $c$ at distance $d(\ell)$ is bounded above by
$$\frac{2}{\frac{1}{\cosh(\ell/4)}},$$
which proves the claim as we are bounding an integer.
\end{proof}
As a consequence, we get the following.
\begin{prop}\label{boundA}
For $S\in\ms_{g,n}$ with $\sys(S)=\ell$
$$|A(S)|\leq \frac{n}{2}\lfloor 2\cosh(\ell/4)\rfloor.$$
\end{prop}
\begin{proof}
There are $n$ cusps, each of which can be surrounded by at most $\lfloor 2\cosh(\ell/4)\rfloor$ cusps at distance $d(\ell)$. The result follows as each curve surrounds two cusps. 
\end{proof}

\begin{rmk}\label{rmk:eulerA}
We can get another upper bound for $A(S)$ using the Euler characteristic as follows.

Consider the set of punctures: if there is a systole bounding two of them, we join them with a simple geodesic lying in the pair of pants determined by the systole. We complete this set of geodesics into an ideal triangulation (decomposition into ideal triangles) of the surface. The number of vertices of the triangulation is the number of punctures $n$. If $e$ is the number of edges, the number of triangles is $\frac{2e}{3}$. The Euler characteristic of the compactified surface is $2g-2$, so
$$n-e+\frac{2e}{3}=2-2g.$$
From how we constructed the triangulation, it is clear that $|A(S)|\leq e$, so we get
$$|A(S)|\leq 3(n+2g-2).$$
Interestingly, this bound can also be seen as a corollary of the above proposition. If we use Schmutz Schaller's upper bound on systole length (Proposition \ref{schmutzsys}) in Proposition \ref{boundA} above, this is exactly the resulting bound.
\end{rmk}

For surfaces of genus at least one, we will use the bound from remark above, but for punctured spheres, we will use Proposition \ref{boundA} directly.

\subsection{Bounds on $|B(S)|$}

Consider a cusp $c$; we define two associated sets:
$$B(c):=\{\alpha\in B(S)\st\exists\beta\in B(S)\mbox{ s.t. $\alpha$ and $\beta$ bound $c$}\}$$
and
$$B(c)^{(2)}:=\{(\alpha,\beta)\in B(S)\times B(S)\st\mbox{ $\alpha$ and $\beta$ bound $c$}\}.$$
Suppose $(\alpha,\beta)$, $(\gamma,\delta)\in B(c)^{(2)}$. Then $\gamma$ has to pass through the pair of pants given by $\alpha$, $\beta$ and $c$, so $\gamma$ must intersect $\alpha$ or $\beta$. Since curves in $\mathfrak{S}(S)\setminus A(S)$ pairwise intersect at most once, $\inter(\alpha,\gamma)=\inter(\beta,\gamma)=1$ (and the same for $\delta$).

Any curve $\alpha\in B(c)$ is at a fixed distance $D(\ell)$ from $H_c$. By a direct computation in the pair of pants bounded by $\alpha$ and $\beta$, one obtains
$$D(\ell)=\log\left(2\frac{\cosh(\ell/2)}{\sinh(\ell/2)}\right).$$

When curves in $B(S)$ intersect they do so exactly once, and we can obtain a lower bound on their angle of intersection of curves in $B(S)$. (Note that the lemma holds for any pair of systoles that intersect once.)

\begin{lemma}\label{boundangle} Let $S$ be a surface of signature $(g,n)\neq (1,1)$. If $\alpha$ and $\beta$ are systoles of length $\ell$ intersecting once, their angle of intersection satisfies
$$\sin\angle(\alpha,\beta)\geq \sin\theta_\ell:=
\left\{\begin{array}{ll}
\frac{2}{\sqrt{5}},& \ell<2\arccosh(3/2)\\
\frac{\sqrt{2\cosh(\ell/2)+1}}{\cosh(\ell/2)+1},&\ell\geq 2\arccosh(3/2).
\end{array}\right.$$
In particular, the angle of intersection is bounded below by a function $\theta_\ell$ that behaves like $e^{-\ell/4}$ as $\ell$ goes to infinity.
\end{lemma}
Note that also \cite[Lemma 2.4]{parlier} gives a lower bound on the angle of intersection, with the same order of growth.
\begin{proof}
Consider the two systoles and the one holed torus $T$ they determine. Since $(g,n)\neq(1,1)$, the boundary component $\delta$ of $T$ is a simple closed geodesic. 

As $\alpha$ and $\beta$ are systoles of $S$, they are also systoles of $T$. As such they satisfy the systole bound for $T$ which depends on the length of $\delta$, namely
$$\cosh(\ell(\delta)/6)\geq \cosh(\ell)-\frac{1}{2}.$$

We first consider the case when $\ell\geq 2\arccosh(3/2)$. We have $\cosh(\ell/2)-\frac{1}{2}\geq 1$ and the condition stated above is non empty. 
Cut $T$ along $\alpha$ and consider the shortest curve $h$ connecting the two copies of $\alpha$.
\begin{figure}[H]
\begin{center}
\includegraphics{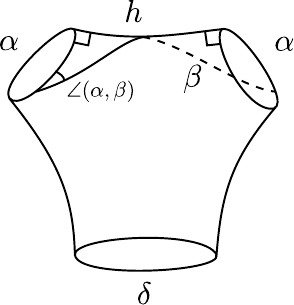}
\caption{The result of cutting the torus $T$ along $\alpha$}\label{angle}
\end{center}
\end{figure}
By hyperbolic trigonometry, using $\cosh(\ell(\delta)/6)\geq \cosh(\ell/2)-\frac{1}{2}$, a direct computation provides
$$\cosh(h)\geq\frac{4\cosh(\ell/2)^2-\cosh(\ell/2)-1}{\cosh(\ell/2)+1}.$$
Now consider one of the two right-angled triangles determined by arcs of $\alpha$, $\beta$ and $h$. We have
$$\frac{\sinh(h/2)}{\sin(\angle(\alpha,\beta))}=\sinh(\ell/2)$$
which, together with the estimate on $h$, yields
$$\sin\angle(\alpha,\beta)\geq \frac{\sqrt{2\cosh(\ell/2)+1}}{\cosh(\ell/2)+1}.$$

If $\ell< 2\arccosh(3/2)$, we deduce the inequality $\sin\angle(\alpha,\beta)\geq\frac{2}{\sqrt{5}}$ by arguing as above, but replacing the estimate $\cosh(\ell(\delta)/6)\geq \cosh(\ell/2)-\frac{1}{2}$ by $\ell(\delta)\geq\ell$.
\end{proof}

Fix $(\alpha,\beta)\in B(c)^{(2)}$ and denote by $\mathcal{P}$ the pair of pants they determine with $c$. As they form a pair of pants with two boundary curves of the same length, there is an isometric involution $\varphi$ of $\mathcal{P}$ that sends $\alpha$ to $\beta$ (a rotation of angle $\pi$ around the cusp). Note that for any $(\gamma,\delta)\in B(c)^{(2)}$, the involution sends $\gamma\cap\mathcal{P}$ to $\delta\cap\mathcal{P}$ because of the symmetry of the pair of pants determined by $\gamma$, $\delta$ and $p$. If we quotient $\mathcal{P}$ by $\varphi$ and we consider the image of $B(c)$, we get a set of geodesics at distance $D(\ell)$ from a horoball of area $1$, all pairwise intersecting with angle at least $\theta_\ell$. This observation is crucial to show the following result.

\begin{lemma}
If $(g,n)\neq (1,1)$, the number of elements in $B(c)$ is bounded above by 
$$
m(\ell):= \frac{\cosh(\ell/2)}{\sinh(\ell/2)}\frac{2}{\sin(\theta_\ell/2)}.
$$
\end{lemma}
\begin{proof}
The situation is as in the following figure which locally represents the elements of $B(c)$ under the quotient by $\varphi$. Note that every element in the quotient by $\varphi$ represents two elements from $B(c)$. 
\begin{figure}[H]
\begin{center}
\includegraphics{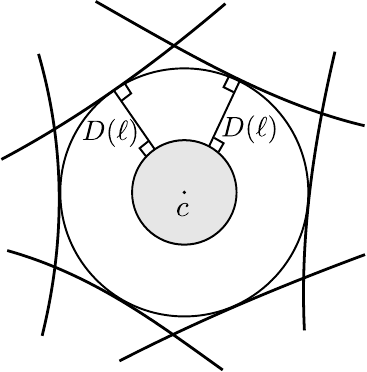}
\caption{Geodesics around a horoball}\label{fig:horoball}
\end{center}
\end{figure}
The inner circle (which we'll refer to as the inner horocycle) represents the quotient horoball of area $1$ and the external one is the horocycle at distance $D(\ell)$ from the horoball of area $1$. By looking at the unique orthogonal geodesics between elements of $B(c)/\varphi$ and the inner horocyle, we can determine a cyclic ordering on the elements of $B(c)/\varphi$. Two neighboring geodesics with respect to this ordering, determine a {\it subarc} on the inner horocycle as follows. We consider the orthogonal geodesic between them and the inner horocycle and take the subarc of the horocyle which forms a pentagon with the two geodesics and the orthogonal (see Figure \ref{fig:horoball}). By a direct computation, using the lower bound on the angle of intersection, this subarc on the inner horocycle is of length at least
$$\frac{\sinh(\ell/2)}{\cosh(\ell/2)}\sin(\theta_\ell/2).$$
These subarcs are all disjoint and are in the same number as the elements of $B(c)/\varphi$ (keep in mind that {\it any} two elements of $B(c)/\varphi$ intersect).

From this we deduce an upper bound on $|B(c)/\varphi|$: 
$$\frac{1}{\frac{\sinh(\ell/2)}{\cosh(\ell/2)}\sin(\theta_\ell/2)}.$$
Now $2\; |B(c)/\varphi| = |B(c)|$, which completes the proof.
\end{proof}
As a consequence, we obtain an upper bound on $|B(S)|$.
\begin{prop}\label{boundB}
If $S\in \ms_{g,n}$, $(g,n)\neq (1,1)$, has systole of length $\sys(S)=\ell$, then
$$|B(S)|\leq n\, m(\ell).$$
\end{prop}
\begin{proof}
We have $$B(S)=\bigcup_{c \mbox{ \scriptsize cusp}}B(c)$$ and for every cusp $c$ $$|B(c)|\leq m(\ell).$$
\end{proof}
\subsection{Bound on $|C(S)|$}
By definition, elements of $C(S)$ are systoles that satisfy
\begin{itemize}
\item two curves in $C(S)$ intersect at most once and
\item two disjoint curves in $C(S)$ do not bound a cusp.
\end{itemize}
We follow a similar argument to one found in \cite{parlier} to obtain an upper bound on $|C(S)|$. In particular we will need a  collar lemma for systoles.

\begin{lemma} Let $\sys(S)=\ell$ and consider $\alpha$, $\beta\in C(S)$. If $\alpha$ and $\beta$ do not intersect, then they are at distance at least $2r(\ell)$, where
$$r(\ell)=\arcsinh\left(\frac{1}{2\sinh(\ell/4)}\right).$$
\end{lemma}
\begin{proof}
Fix a pair of pants with $\alpha$ and $\beta$ as boundary and consider the  third boundary component $\gamma$. Since $\alpha$ and $\beta$ are in $C(S)$, they do not bound a cusp, so $\gamma$ is a simple closed geodesic of length at least $\ell$. The result follows by a standard trigonometric computation.
\end{proof}
As a consequence, if $\alpha$ and $\beta$ in $C(S)$ pass through the same disk of radius $r(\ell)$ then they intersect.

Moreover, we have seen in Lemma \ref{boundangle} that there is a lower bound on the angle of intersection of systoles intersecting once. With this in hand we prove the following.
\begin{lemma}
If $(g,n)\neq (1,1)$, $\sys(S)=\ell$ and $\alpha$ and $\beta$ in $C(S)$ pass through a disk of center $p$ and radius $r(\ell)$, the distance between $p$ and the point $q$ of intersection between $\alpha$ and $\beta$ satisfies
$$d(p,q)\leq R(\ell),$$
where
$$
\sinh(R(\ell))=\left\{\begin{array}{ll}
\frac{5}{8\sinh(\ell/4)},& \ell<2\arccosh(3/2)\\
\frac{\cosh(\ell/2)+1}{2\sinh(\ell/4)\sqrt{2\cosh(\ell/2)+1}},&\ell\geq 2\arccosh(3/2).
\end{array}\right.$$
\end{lemma}

Note that $R(\ell)$ is bounded for $\ell\geq 2\arcsinh(1)$.

\begin{proof}
The proof is analogous to the proof of \cite[Lemma 2.6]{parlier}. Fix $p_\alpha\in \alpha$ and $p_\beta\in \beta$ lying in $D_{r(\ell)}(p)$. We have two triangles of vertices $p, p_\alpha, q$ and $p, p_\beta, q$, and the sum of the two angles $\theta_\alpha$ and $\theta_\beta$ at $q$ is the angle of intersection $\angle(\alpha,\beta)$. Suppose $\theta_{\alpha}\geq \frac{\angle(\alpha,\beta)}{2}$ and consider the angle $\eta$ of the triangle $p, p_\alpha, q$ at $p_\alpha$.
\begin{figure}[H]
\begin{center}
\includegraphics{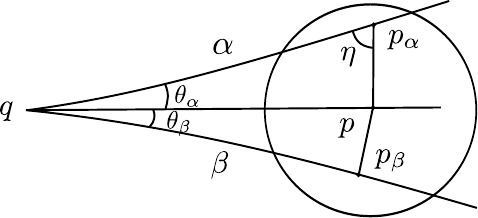}
\caption{$\alpha$ and $\beta$ passing though a disk of radius $r(\ell)$}
\end{center}
\end{figure}
Then
$$\frac{\sin(\eta)}{\sinh(d(p,q))}=\frac{\sin(\theta_\alpha)}{\sinh(d(p,p_\alpha))}.$$
Using $\theta_{\alpha}\geq \frac{\angle(\alpha,\beta)}{2}$, $d(p,p_\alpha)<r(\ell)$ and Lemma \ref{boundangle}, we obtain the claimed result.
\end{proof}

We are now in a position to obtain a bound on $|C(S)|$.

\begin{prop}\label{boundC}
If $S\in \ms_{g,n}$, $g\neq 0$ and $(g,n)\neq (1,1)$, has systole of length $\sys(S)=\ell$, then
$$|C(S)|\leq 200\,\frac{e^{\ell/2}}{\ell}(2g-2+n).$$
\end{prop}

\begin{proof}
If $\ell\leq 2\arcsinh(1)$, then all systoles are pairwise disjoint, so
$$|C(S)|\leq \kiss(S)\leq 3g-3+n.$$

We now suppose that $\ell>2\arcsinh(1)$. Consider $\tilde{S}=S\setminus \bigcup_{c \mbox{  \scriptsize cusp}}D_{w(\ell)}(c)$, where $$w(\ell)=\arcsinh\left(\frac{1}{\sinh(\ell/2)}\right)$$ is the width of a collar around a systole. By the  collar lemma, each curve of $C(S)$ is contained in $\tilde{S}$. We cover $\tilde{S}$ with disks of radius $r(\ell)$. Then the cardinality of $C(S)$ is bounded above by $$\frac{F(S)G(S)}{H(S)},$$
where
\begin{align*}
F(S)&=\#\{\mbox{balls of radius $r(\ell)$ needed to cover }  \tilde{S}\}\\
G(S)&=\#\{\mbox{curves in $C(S)$ crossing a ball of radius }r(\ell)\}\\
H(S)&=\#\{\mbox{number of balls of radius $r(\ell)$ a curve in $C(S)$ must cross}\}
\end{align*}
To bound $|C(S)|$, we need to give upper bounds for $F(S)$ and $G(S)$ and a lower bound for $H(S)$.

\underline{Upper bound for $F(S)$}

We have
\begin{gather*}
F(S)\leq \max \#\{\mbox{embedded balls of radius $r(\ell)/2$ which are pairwise disjoint}\}\leq \\
\leq \frac{\area(\tilde{S})}{\area(\mbox{\small ball of radius }r(\ell)/2)}\leq \frac{\area(S)}{2\pi (\cosh(r(\ell)/2)-1)}\leq 8(2g-2+n)e^{\ell/2}.
\end{gather*}

\underline{Upper bound for $G(S)$}

We proceed as in the proof of Theorem 2.9 in \cite{parlier}, by reasoning in the universal cover and estimating how many geodesics, pairwise intersecting at an angle of at least $\theta_\ell$, can intersect a disk of radius $r(\ell)$. We obtain
$$G(S)\leq \frac{\pi}{2}\frac{\sinh(R(\ell)+\arcsinh(1))}{\arcsinh(\sin(\theta_\ell))}\leq \frac{5\pi}{2\arcsinh(\sin(\theta_\ell))}.$$

\underline{Lower bound for $H(S)$}

To cover a curve of length $\ell$ with disks of radius $r(\ell)$ we need at least $\frac{\ell}{2r(\ell)}$. So
$$H(S)\geq \frac{\ell}{2\arcsinh\left(\frac{1}{2\sinh(\ell/4)}\right)}\geq \ell\sinh(\ell/4).$$

By putting the three bounds together and considering that $\sinh(\ell/4)\arcsinh(\sin(\theta_\ell))$ is bounded below by $1/3$ for $\ell>2\arcsinh(1)$ we obtain the claimed result.
\end{proof}

\subsection{Proof of main results}

Using Propositions \ref{boundA}, \ref{boundB} and \ref{boundC}, we get an upper bound for the kissing number of a surface in terms of its signature and of its systole length.

\begin{theor}\label{kissbound}
If $S\in \ms_{g,n}$ ($g\geq 1$, $(g,n)\neq (1,1)$) has systole of length $\sys(S)=\ell$, then
$$\kiss(S)\leq 20\, n\cosh(\ell/4)+200\,\frac{e^{\ell/2}}{\ell}(2g-2+n).$$
\end{theor}

As a consequence, we can get a bound on the kissing number which is independent on the systole length.

\begin{theor}\label{kissbound2}
There exists a universal constant $C$ (which we can take to be $2\times 10^4$) such that for any $S\in \ms_{g,n}$, $g\geq 1$, its kissing number satisfies
$$\kiss(S)\leq C\,(g+n)\frac{g}{\log(g+1)}.$$
\end{theor}
\begin{proof}
It follows from the bounds in Theorem \ref{kissbound} and bounds on systole lengths. Precisely we insert the Schmutz Schaller bound (Theorem \ref{schmutzsys}) in the term $\cosh(\ell/4)$ and we use Theorem \ref{sysbound} for the $e^{\ell/2}/\ell$ term. For $(g,n)=(1,1)$,  we recall the well known fact that $\kiss(S)\leq 3$ (there can be at most $3$ distinct curves that pairwise intersect at most once on a one-holed torus). 
\end{proof}

\begin{rmk}
In \cite{przytycki}, Przytycki obtained an upper bound for the number of simple closed curves pairwise intersecting at most once. Using this, and our bound on $|A(S)|$ (Proposition \ref{boundA}), one can obtain an upper bound for the kissing number which is cubic in the Euler characteristic. Our upper bound, on the other hand, is sub-quadratic in $|\chi(S)|$, like the one for closed surfaces in \cite{parlier}. 
\end{rmk}

The upper bound of Theorem \ref{kissbound2} is linear in the number of cusps, if we fix the genus. For punctured spheres we can obtain a more meaningful bound.
\begin{theor}
For every $S\in\ms_{0,n}$, the number of systoles satisfies
$$\kiss(S)\leq \frac{7}{2}n-5.$$
\end{theor}

\begin{proof}
By Proposition \ref{boundA} and Schmutz Schaller's upper bound for the systole, we have
$$|A(S)|\leq \frac{n}{2}\left\lfloor\frac{2(3n-6)}{n}\right\rfloor=\frac{n}{2}\left\lfloor 6-\frac{12}{n}\right\rfloor\leq\frac{5}{2}n.$$
Moreover, systoles are separating, so can only intersect an even number of times. This implies that systoles in $\mathfrak{S}(S)\setminus A(S)$ are pairwise disjoint and hence part of a pants decomposition. Note that any pants decomposition of a sphere contains at least two curves bounding two cusps: indeed, the dual graph to the pants decomposition is a tree, so it has at least two leaves, which correspond to curves bounding two cusps. This implies that
$$|\mathfrak{S}(S)\setminus A(S)|\leq \#\mbox{curves in a pants decomposition}-2=n-5.$$
\end{proof}

By using short pants decompositions where every curve is of equal length, it is easy to obtain a family of punctured spheres with a number of systoles that grows linearly in the number of cusps. Matching the $\frac{7}{2}n$ upper bound from this theorem seems much more challenging. 

\bibliographystyle{alpha}
\bibliography{references}
\end{document}